\documentclass{article}

\usepackage{amsmath}
\usepackage{amssymb}
\usepackage{amsthm}

\newtheorem{theorem}{Theorem}
\newtheorem{lemma}[theorem]{Lemma}
\newtheorem{cor}[theorem]{Corollary}

\theoremstyle{definition}
\newtheorem{prop}[theorem]{Proposition}
\newtheorem{definition}{Definition}

\theoremstyle{remark}

\newtheorem*{remark}{Remark}

\newcommand{\bv}{\mathbf{v}}
\DeclareMathOperator{\diag}{diag}
\DeclareMathOperator{\spn}{span}

\title{Positive Semidefiniteness of Matrices arising from Ramsey Theory}

\author{Joshua Cooper and Maxwell Forst}

\date{\today}

\begin{document}
\maketitle

\begin{abstract}
We resolve a conjecture from \cite{CFP11} that a certain sequence of combinatorial matrices which can be used to bound small product-Ramsey numbers is positive semidefinite.  Because the connection to Ramsey Theory involves solving quadratic integer programs associated to these matrices, this implies that there are relatively efficient algorithms for bounding said numbers.  The proof is direct, and yields important structural information: we enumerate the eigenvalues and eigenspaces explicitly by employing hypergeometric identities.
\end{abstract}

\section{Introduction}

In \cite{CFP11}, the authors consider the following question about small-parameter Product Ramsey numbers: Given positive integers \(c\) (the ``number of colors'') and \(d\) (the ``dimension''), for which values of \(a_1,\ldots,a_d \in \mathbb{N}\) is it true that every \(c\)-coloring of the \(a_1 \times \cdots \times a_d\) grid admits at least one monochromatic \(2 \times \cdots \times 2\) subgrid?  To be precise, we say that \([a_1,\ldots,a_d]\) is {\em \(c\)-guaranteed} if, for every function \(\eta:[a_1] \times \cdots \times [a_d] \rightarrow [c]\), there are \(d\) pairs \(\{x_i,y_i\}\) with \(x_i,y_i \in [a_i]\) for each \(i \in [d]\), so that \(\eta\) restricted to \(\{x_1,y_1\} \times \cdots \{x_d,y_d\}\) is constant.  Which tuples \(a_1,\ldots,a_d\) are \(c\)-guaranteed?  The set of such \(c\)-guaranteed \(d\)-tuples is an up-set in the natural coordinate-wise ordering (i.e., Cartesian product of total orders) on \(\mathbb{N}^d\), and its minimal elements are a finite antichain, the {\em obstruction set} for parameters \(c\) and \(d\).

The question of describing the obstruction set for small values of \(c\) and \(d\) is an interesting one, considered \cite{FGGP12} and \cite{AGL12} in addition to \cite{CFP11}.  The smallest nontrivial case is \(c = 2\), and for \(d = 2\), \cite{FGGP12} gives a complete description of the obstruction set: \(\{(3,7),(5,5),(7,3)\}\).  For \(c=2\) and \(d=3\), the situation is already not completely understood, although \cite{CFP11} offers some answers.  In particular, the authors show the following.

\begin{prop} Write \(\mathcal{P}(S)\) for the power set of \(S\), and suppose \(r,s > 1\).  Let \(M_r\) denote the \(2^r \times 2^r\) square matrix indexed by the sets \(S,T \subset [r]\) defined by
\[
M_r(S,T)=\binom{|S \cap T|}{2}+\binom{|\bar{S} \cap \bar{T}|}{2}.
\]
Let \(t\) be the least value of \(\bv^\ast M_r \bv - \bv \cdot \diag(M_r)\) over all vectors \(\bv \in \mathbb{N}^{\mathcal{P}([r])}\) with \(\sum_j \bv_j = s\).  Then \([r,s,\lfloor(r(r-1)s(s-1)/(2t)\rfloor + 1]\) is \(2\)-guaranteed.
\end{prop}

This provides (rather decent) upper bounds on the \(c=2\), \(d=3\) obstruction set via a simple quadratic integer programming computation.  Such computation can use polynomial time convex programming techniques during the interior point search if \(M_r\) is positive semidefinite -- which \cite{CFP11} claimed it is for \(r \leq 9\) via direct computation.  The authors conjecture that \(M_r\) is positive semidefinite for all \(r\), in fact.  That is a consequence of the main result of our paper, which gives a complete description of the eigenvalues and eigenspaces associated to \(M_r\).

\begin{theorem} \label{thm:main-thm} Let \(r \geq 4\). The eigenvalues of \(M_r\) are \(0\), \(2^{r-3}(r-2)\), \(2^{r-3}\), \(2^{r-4}(r-1)(r-2)\), and \(1\), with multiplicities \(2^r - \binom{r+1}{2}\), \(r-1\), \(\binom{r}{2} - 1\), \(1\), and \(1\), respectively.
\end{theorem}

\begin{cor} \(M_r\) is positive semidefinite for every \(r \geq 0\).
\end{cor}

Our proof strategy is straightforward: we describe the eigenspaces explicitly, show they correspond to the claimed eigenvalues, compute their dimensions, and conclude when the dimensions add up to \(2^r\).  Note that several standard combinatorial identities are employed; their statements are relegated to the Appendix.

We remark that the problem of characterizing the eigenvalues and eigenspaces of \(M_r\) can be generalized in many ways, none of which we have pursued here.  For example, one might consider \(M_{r,k}\), indexed by subsets of \([r]\), with \((S,T)\)-entry equal to \(\binom{|S \cap T|}{k} + \binom{|\bar{S} \cap \bar{T}|}{k}\).  The following simple argument establishes at least semi-definiteness in even more general circumstances.

\begin{prop} \label{prop:setsystem} Let \(Y_1,\ldots,Y_n\) be subsets of a common ground set \(Y\). Let rows and columns of the symmetric \(n \times n\) matrix \(A\) be indexed
 by these subsets, and let entry \(A(Y_i,Y_j) = |Y_i \cap Y_j|\). Then matrix \(A\) is positive semi-definite. 
\end{prop}

\begin{proof} Let \(I_y : \mathcal{P}(Y) \rightarrow \{0,1\}\) be the indicator function of the membership of \(y \in Y\).  Then, for an arbitary vector \(\mathbf{x} \in \mathbb{C}^{n}\),
\begin{align*}
\mathbf{x}^\ast A \mathbf{x} &= \sum_{i=1}^n \sum_{j=1}^n A(Y_i,Y_j) \overline{x_i} x_j \\
&= \sum_{i=1}^n \sum_{j=1}^n \sum_{y \in Y} I_y(Y_i) I_y(Y_j) \overline{x_i} x_j \\
&= \sum_{y \in Y} \left | \sum_i I_y(Y_i) x_i \right |^2 \geq 0.
\end{align*}
\end{proof}

To apply this to \(M_r\), let \(Y\) be the set of 2-element subsets of \([r]\), and let the sets \(Y_j\) be all subsets of \(Y\) consisting of the 2-element subsets of any subset \(S_j \subseteq [r]\).  (That is, \(Y\) is the edge sets of a complete graph on \(r\) vertices, and each \(Y_j\) is the edge set of a clique therein.)  Then \(M_r = A + A^T\) if \(A = A(Y_i,Y_j)\) is the matrix given by Proposition \ref{prop:setsystem}.

In the sequel, we use some of the following notation.  We always assume \(r \geq 4\), as our main result is simple to check for \(r < 4\). The set \(\{1,\ldots,r\}\) is denoted by \([r]\); the complement (with respect to \([r]\)) of a set \(S\) is written \(\overline{S}\).  We employ the conventions that lowercase Latin letters denote nonnegative integers; \(S\), \(T\), and \(U\) denote subsets of \([r]\), which we also think of as indicator functions \([r]\rightarrow \{0,1\}\), so that \(A(x) = 1\) if \(x \in A\) and \(A(x) = 0\) if \(x \not \in A\); \(V\) denotes a vector indexed by subsets of \([r]\) (i.e., \(V \in \mathbb{Z}^{\mathcal{P}([r])}\)); and calligraphic letters (other than \(\mathcal{P}\), which denotes the power set operator) denote families of vectors.  Generally, if \(V \in \mathbb{F}^{\mathcal{I}}\) is a vector and \(\alpha \in \mathcal{I}\) is an index, then we denote the \(\alpha\)-coordinate of \(V\) by \(V_\alpha\); if \(A \in \mathbb{F}^{\mathcal{I} \times \mathcal{J}}\) is a matrix and \((\alpha,\beta) \in \mathcal{I}\times\mathcal{J}\), then we denote the \((\alpha,\beta)\)-entry of \(A\) by \(A(\alpha,\beta)\).


\section{Eigenvector families of \(M_r\)}
We begin by exhibiting seven families of eigenvectors of \(M_r\) and computing their corresponding known eigenvalues.


\subsection{Eigenvector families of \(N_r\)}
In order to show eigenvalues of the first three eigenvector families we first define an auxiliary ``quotient'' matrix \(N_r\).

\begin{definition}
   Let \(s,t \in \{0,\ldots,r\}\). We define \(N_r\) to be a matrix indexed by \(s\) and \(t\) such that
   \[N_r(s,t)=\sum_{T:|T|=t}M_r([s],T).\]
\end{definition}

\begin{lemma}
   \[N_r(s,t)=\binom{s}{2}\binom{r-2}{r-t}+\binom{r-s}{2}\binom{r-2}{t}.\]
\end{lemma}

\begin{proof}
   \begin{align*}
      N_r(s,t)&=\sum_{T:|T|=t}M_r([s],T)\\
      &=\sum_{\substack{a,b \\ r=s+t-a+b}} \left ( \binom{a}{2}+\binom{b}{2} \right ) \mu (t;a,b),
   \end{align*}
   where \(a, b \in \{0, \ldots, r\}\) and \(\mu (t;a,b)\) is the number of ways to choose \(t\) given the constraints \(a = |[s] \cap T|\), \(b = |\bar{[s]} \cap \bar{T}|\).
   \begin{align*}
      \mu (t;a,b) &= \binom{|[s]|}{|[s] \cap T|}
      \binom{|[r] \backslash S|}{|\bar{[s]} \cap \bar{T}|}\\
      &= \binom{s}{a}\binom{r-s}{b}.
   \end{align*}
   Therefore,
   \begin{align*}
      N_r(s, t) &= \sum_{\substack{a,b \\ r=s+t-a+b}} \left ( 
      \binom{a}{2}+\binom{b}{2} \right )\binom{s}{a}\binom{r-s}{b} \\ 
      &=\sum_{\substack{a,b \\ r=s+t-a+b}} \binom{a}{2}\binom{s}{a}\binom{r-s}{b}
      + \sum_{\substack{a,b \\ r=s+t-a+b}} \binom{b}{2}\binom{s}{a}\binom{r-s}{b}.
   \end{align*}
   By (\ref{eq1}) and (\ref{eq3}),
   \begin{align*}
      \sum_{\substack{a,b \\ r=s+t-a+b}} \binom{a}{2}\binom{s}{a}\binom{r-s}{b}
      &=\binom{s}{2} \sum_{\substack{a,b \\ r=s+t-a+b}} \binom{s-2}{a-2}\binom{r-s}{b}\\
      &=\binom{s}{2} \sum_b \binom{s-2}{r-t-b}\binom{r-s}{b},\\
	  &=\binom{s}{2}\binom{r-2}{r-t}.
   \end{align*}
   By a similar argument,
   \[\sum_{\substack{a,b \\ r=s+t-a+b}} \binom{b}{2}\binom{r-s}{b}\binom{s}{a}
   =\binom{r-s}{2}\binom{r-2}{t}.\]
   Therefore,
   \[N_r(s,t)=\binom{s}{2}\binom{r-2}{r-t}+\binom{r-s}{2}\binom{r-2}{t}\]
\end{proof}


\begin{definition}
Let \(\mathcal{V}'_r\) be the set of all vectors indexed by \(\{0,\ldots,r\}\) (i.e., \(\mathbb{R}^{r+1}\)) and \(\mathcal{V}_r\) be the set of all vectors indexed by \(\mathcal{P}([r])\) (i.e., \(\mathbb{R}^{\mathcal{P}([r])}\)) . We define \(\Phi : \mathcal{V}'_r \rightarrow \mathcal{V}_r\) by
   \[(\Phi(V))_T=V_{|T|}\]
for each \(T \subset [r]\).
\end{definition}

\begin{lemma} \label{lem:shared ev}
   Let \(V\) be an eigenvector of \(N_r\). Then \(\Phi (V)\) is an eigenvector of \(M_r\) with the same eigenvalue.
\end{lemma}

\begin{proof}
   Suppose \(V\) is an eigenvalue of \(N_r\) with eigenvalue \(\lambda\). Then, for each \(S \subseteq [r]\),
   \begin{align*}
      \left ( M_r \Phi (V) \right )_S &= \sum_{T \subseteq [r]} M_r(S,T) (\Phi (V))_T\\
      &=\sum_{|T|=0}^r \sum_{T} M_r(S,T) (\Phi (V))_T\\
      &=\sum_{|T|=0}^r N_r(|S|,|T|)V_T\\
      &= ( N_r V)(|S|) = \lambda V(|S|) = \lambda (\Phi (V))_S
   \end{align*}
\end{proof}


Next, we use Lemma \ref{lem:shared ev} to describe three eigenvector families of \(M_r\) arising from eigenvectors of \(N_r\).  The first family, \(\mathcal{A}_r\), has only one element.

\begin{prop}
Let \(\mathcal{A}_r\) be the (singleton) set whose one element is \(\Phi(V)\), where \(V\) is the vector with
   \[V_t = 
      \begin{cases}
         1 & t=r-1\\
         2-r & t=r\\
         0 & t \leq r-2.
      \end{cases}\]
for \(t \in \{0,\ldots,r\}\).  Then elements of \(\mathcal{A}_r\) are eigenvectors of \(M_r\) associated with eigenvalue \(0\).
\end{prop}

\begin{proof}
   For any \(s \in \{0, \ldots, r\}\),
   \begin{align*}
      (N_r V)_s &=N_r(s, r-1) - (r-2)N_r (s, r)\\
      &= \binom{s}{2}\binom{r-2}{1} - (r-2) \binom{s}{2}
      \binom{r-2}{0} \\
      &=(r-2)\binom{s}{2}(1-1)=0.
   \end{align*}
   Therefore \(N_r V = 0 \cdot V\). By Lemma \ref{lem:shared ev}, this implies the elements of \(\mathcal{A}_r\) are eigenvectors of \(M_r\) with eigenvalue \(0\). 
\end{proof}


The second family, \(\mathcal{B}_r\), is also a singleton.

\begin{prop}
   Let \(\mathcal{B}_r\) be the (singleton) set whose one element is \(\Phi (V)\), where \(V\) is a vector defined by
   \[V_t=\binom{t}{2}+\binom{r-t}{2}.\]
   for \(t \in \{0,\ldots,r\}\).  Then the elements of \(\mathcal{B}_r\) are eigenvectors of \(M_r\) associated with eigenvalue \(2^{r-4}(r^2-r+2)\).
\end{prop}

\begin{proof}
   For any \(s \in \{0, \ldots, r\}\),
   \begin{align*}
      (N_r V)_s &= \sum_t \left[ \binom{s}{2}\binom{r-2}{r-t}
      + \binom{r-s}{2}\binom{r-2}{t}\right] \left[ \binom{t}{2} 
      + \binom{r-t}{2} \right]\\
      &=\binom{s}{2}P + \binom{r-s}{2}Q,
   \end{align*}
   where
   \begin{align*}
      P &= \sum_t \left[ \binom{t}{2}\binom{r-2}{t-s} 
      + \binom{r-t}{2}\binom{r-2}{r-t}\right]\\
      Q &= \sum_t \left[ \binom{t}{2}\binom{r-2}{t} 
      + \binom{r-t}{2}\binom{r-2}{r-t-2}\right].
   \end{align*}
   By (\ref{eq5}),
   \begin{align*}
      P &= 2^{r-5}[(r-2)(r-3) + 8(r-2) + 8 + (r-2)(r-3)]\\
      &=2^{r-4}(r^2-r+2)
   \end{align*}
   and
   \begin{align*}
      Q &= 2^{r-5}[(r-2)(r-3) +(r-2)(r-3) + 8(r-2) +8]\\
      &=2^{r-4}(r^2-r+2).
   \end{align*}
   Therefore,
   \begin{align*}
      (N_r V)_s &= 2^{r-4}(r^2-r+2) \left[\binom{s}{2} + \binom{r-s}{2} \right]\\
      &=2^{r-4}(r^2-r+2)V_s.
   \end{align*}
   Thus, \(N_r V = 2^{r-4}(r^2-r+2)V\), which by Lemma \ref{lem:shared ev} implies that the elements of \(\mathcal{B}_r\) are eigenvectors of \(M_r\) with eigenvalue \(2^{r-4}(r^2-r+2)\).
\end{proof}


The next family, \(\mathcal{C}_r\), is also a singleton.

\begin{prop}
   Let \(\mathcal{C}_r\) be the (singleton) set whose one element is \(\Phi (V)\), where \(V\) is a vector defined by
   \[V_t=\binom{t}{2}-\binom{r-t}{2}.\]
   Then the elements of \(\mathcal{C}_r\) are eigenvectors of \(M_r\) with eigenvalue \(2^{r-2}(r-1)\).
\end{prop}

\begin{proof}
   For any \(s \in \{0, \ldots, r\}\),
   \begin{align*}
      (N_r V)_s &= \sum_t \left[ \binom{s}{2}\binom{r-2}{r-t}
      + \binom{r-s}{2}\binom{r-2}{t}\right] \left[ \binom{t}{2} 
      - \binom{r-t}{2} \right]\\
      &=\binom{s}{2}P + \binom{r-s}{2}Q,
   \end{align*}
   where
   \begin{align*}
      P &= \sum_t \left[ \binom{t}{2}\binom{r-2}{t-s} 
      - \binom{r-t}{2}\binom{r-2}{r-t}\right]\\
      Q &= \sum_t \left[ \binom{t}{2}\binom{r-2}{t} 
      - \binom{r-t}{2}\binom{r-2}{r-t-2}\right].
   \end{align*}
   By (\ref{eq5}),
   \begin{align*}
      P &= 2^{r-5}[(r-2)(r-3) + 8(r-2) + 8 - (r-2)(r-3)]\\
      &=2^{r-2}(r-1)
   \end{align*}
   and
   \begin{align*}
      Q &= 2^{r-5}[(r-2)(r-3) - ((r-2)(r-3) + 8(r-2) +8)]\\
      &=-2^{r-2}(r-1).
   \end{align*}
   Therefore,
   \begin{align*}
      (N_r V)_s &= 2^{r-2}(r-1) \left( \binom{s}{2} - \binom{r-s}{2} \right)\\
      &=2^{r-2}(r-1)V_s.
   \end{align*}
   Thus, \(N_r V = 2^{r-2}(r-1)V\), which by Lemma \ref{lem:shared ev} implies that all elements of \(\mathcal{C}_r\) are eigenvectors of \(M_r\) with eigenvalue \(2^{r-2}(r-1)\).
\end{proof}


\subsection{Eigenvector Families of \(M_r\)}
In this section we compute \((M_r V)_S\) for various vectors \(V \in \mathbb{R}^{\mathcal{P}([r])}\) and sets \(S \subseteq [r]\), using the definition of matrix multiplication.  Each calculation resembles the following template:
\begin{align*}
   (M_r V)_S &= \sum_T M_r(S,T) V_T\\
   &= \sum_T \left[ \binom{|S \cap T|}{2} + \binom{|\bar{S} \cap \bar{T}|}{2}
   \right]V_T\\
   &= \sum_{\substack{a,b \\ a = |S \cap T| \\ b = |\bar{S} \cap \bar{T}|}}
   \left[ \binom{a}{2} + \binom{b}{2} \right]V_T\\
   &= \sum_{a,b} \left[ \binom{a}{2} + \binom{b}{2} \right] \chi (S;a,b) \cdot V,
\end{align*}
where \(\chi(S;a,b) \in \mathbb{R}^{\mathcal{P}([r])}\) is a vector indexed by the subsets of \([r]\), with coordinates given by
\[\chi(S;a,b) = \begin{cases}
1 & \textrm{ if } |T| = r-|S|+a-b\\
0 & \textrm{ if } |T| \neq r-|S|+a-b.
\end{cases}\]
We then show that 
\[\sum_{a,b} \left[ \binom{a}{2} + \binom{b}{2} \right]   \chi (S;a,b) \cdot V = \lambda(r) V\]
for all \(S\) where \(\lambda(r)\) is a constant depending only on \(r\).

%

\begin{prop}\label{prop:Dvecs}
Let \(\mathcal{D}_r\) be the set of all vectors \(V\) indexed by subsets \(T \subset [r]\), with
   \[V_T =
   \begin{cases}
   0 & U \not\subseteq T\\
   (-1)^{|T|} & U \subseteq T
   \end{cases}\]
   for some fixed \(U \subseteq [r]\) such that \(|U| < r-2\). Then all elements of \(\mathcal{D}_r\) are eigenvectors of \(M_r\) with eigenvalue \(0\). 
\end{prop}

We fix such a \(U\) and proceed to show the corresponding \(V\) is an eigenvector.  A few lemmas will simplify the proof.

   \begin{lemma} \label{lemma:firststepforD}
      \[\chi (S;a,b) \cdot V = (-1)^{r-|S|+a-b} \binom{|S|-|S \cap U|}{a-|S \cap U|} \binom{r-|S|-|U|+|S \cap U|}{b}\]
   \end{lemma}
   
   \begin{proof}
      For ease of notation we define \( S \thicksim T\) to mean \(|S \cap T| = a\) and \(|\bar{S} \cap \bar{T}| = b\). Note that \(S \thicksim T\) implies that \(r = |S|+|T|-|S\cap T|+ |\bar{S} \cap \bar{T}|\).
      \[\chi (S;a,b) \cdot V = \sum_{\substack{T \\ S \thicksim T\\ U \subseteq T}} (-1)^{|T|}=(-1)^{|T|}\sum_{\substack{T \\ S \thicksim T \\ U \subseteq T}} 1.\]
      This implies that the value of \(\chi (S;a,b) \cdot V\) is given by \((-1)^{|T|}\) times the number of ways to choose \(T\) such that \(|S \cap T| = a\), \(|\bar{S} \cap \bar{T}| = b\) and \(U \subseteq T\). Let \(T_1=T \cap ( S \backslash U )\) and \(T_2 = T \cap ( \bar{S} \backslash U) \). \(T_1\) consists of \(a-|S \cap U|\) elements chosen from a set of \(|S \backslash U|\) elements and \(T_2\) consists of \(b\) elements chosen from a set of \(|\bar{S} \backslash U|\) elements. Therefore
      \begin{align*}
      \chi (S;a,b) \cdot V &= (-1)^{r-|S|+a-b}|T_1||T_2| \\
      &= (-1)^{r-|S|+a-b}\binom{|S \backslash U|}{a - |S \cap U|}
      \binom{|\bar{S} \backslash U|}{b}\\
      &= (-1)^{r-|S|+a-b}\binom{|S|-|S \cap U|}{a-|S \cap U|} \binom{r-|S|-|U|+|S \cap U|}{b}
      \end{align*}
   \end{proof}
   
   Using the template argument above, we can rewrite Lemma \ref{lemma:firststepforD} as 
   \[
	(M_r V)_S =\sum_{a,b} \left[ \binom{a}{2}+\binom{b}{2} \right] (-1)^{r-|S|+a-b} 	\binom{s-|U'| }{a - |U'|}\binom{r-s-|U|+|U'|}{b}.
	\]
	the terms of which we group as
	\begin{equation} \label{eqPQs}
	(M_r {V})_S =(PQ+P'Q')(-1)^{r-|S|},
	\end{equation}
	where
    \begin{align*}
		P &= \sum_a (-1)^a \binom{a}{2}\binom{|S|-|U'|}{a-|U'|}\\
		Q &= \sum_b (-1)^b \binom{r-|S|-|U|+|U'|}{b}\\
		P' &= \sum_a (-1)^a \binom{|S|-|U'|}{a-|U'|}\\
		Q' &= \sum_b (-1)^b \binom{b}{2}\binom{r-|S|-|U|+|U'|}{b}.\\
	\end{align*}
	\begin{lemma}\label{lemma:PQ}
		\(PQ = 0\) if \(|U| < r-2\)
	\end{lemma}

	\begin{proof}
		\(Q=0\) if \(r-|S|-|U|+|U'| \neq 0\).
		By (\ref{eq1})
        \begin{align*}
			P =& \sum_a (-1)^a \left [\binom{|S|-|U'|}{2}\binom{|S|-|U'|-2}
            	{a-|U'|-2} \right .\\ 
			&\left . + (a-|U'|)|U'|\binom{|S|-|U'|}{a-|U'|}+\binom{|U'|}
            	{2}\binom{|S|-|U'|}{a-|U'|} \right ].
		\end{align*}
		If \(|S|-|U'|>0\) then 
        \[
        \binom{|S|-|U'|}{2} \sum_a (-1)^a \binom{|S|-|U'|-2}{a-|U'|-2} = 0.
        \]
        Likewise, if \(|S|-|U'|>0\) then
        \[
        \binom{|U'|}{2} \sum_a (-1)^{a}\binom{|S|-|U'|}{a-|U'|}=0.
        \]
        If \(|S|-|U'|>1\) then
        \[
        |U'| \sum_a (-1)^{a}(a-|U'|)\binom{|S|-|U'|}{a-|U'|}=0.
        \]
        Therefore, if \(|S|-|U'| > 2\) then  \(P = 0\). If \(Q\) is not equal to \(0\) then \(r-|S|-|U|+|U'|=0\) and we substitute to get \(PQ = 0\) if \(r-2>|U|\).
	\end{proof}

	\begin{lemma}\label{lemma:P'Q'}
		If \(|U'| < r-2\) then \(P'Q' = 0\).
	\end{lemma}

	\begin{proof}
		If \(|S| \neq |U'|\) then \(P' = 0\)
		\begin{align*}
			Q'&=\sum_b (-1)^b \binom{b}{2}\binom{r-|S|-|U|+|U'|}{b}
            \\ &= \binom{r-|S|-|U|+|U'|}{2}\sum_b (-1)^b \binom{r-|S|-|U|+|U'|-2}{b-2}.
		\end{align*}
        If \(r-|S|-|U|+|U'|>2\) then \(Q' = 0\). Therefore, if \(|U| < r-2\) then \(P'Q' = 0 \).
	\end{proof}

\begin{proof}[Proof of Proposition \ref{prop:Dvecs}] As \(|U|<r-2\) by assumption, applying Lemmas \ref{lemma:PQ} and \ref{lemma:P'Q'} to (\ref{eqPQs}), we conclude that \((M_r {V})_S = 0\) for all \(S\).

\end{proof}
%
%
\begin{prop} \label{prop:x-y}
	Let \(\mathcal{E}_r\) be the set of all vectors of the form \(V = V^x -V^y\), where \(x,y \in [r]\) and \(V^z_T = T(z)\). Then all elements of \(\mathcal{E}_r\) are eigenvectors of \(M_r\) with eigenvalue 
  \(2^{r-3}(r-2)\).
	\end{prop}

We begin with a lemma.

		\begin{lemma} \label{x-y:chi dot Vx}
		\[
		\chi (S;a,b) \cdot {V}^x =\binom{|S|- S(x) }{a- S(x)}\binom{r-|S|-1+S(x)}{b}
		\]
	\end{lemma}

	\begin{proof} Clearly,
		\[
		\chi(S;a,b) \cdot {V}^x = \sum_{\substack{T \ni x \\ S \thicksim T }} 1
		\]
		Therefore \(\chi(S;a,b) \cdot {V}^x\) is given by the number of ways to chose \(T\) such that \(|S \cap T| = a\), \(|\bar{S} \cap \bar{T}|=b\), and \(x \in T\). Let \(T_1 = (S\ \backslash \{x\}) \cap T\) and \(T_2 = (\bar{S} \backslash \{x\}) \cap T \). Then the number of ways to chose \(T_1\) is given by \(\binom{|S \backslash \{x\}|}{a-|S \cap \{x\}|}\) and the number of way to chose \(T_2\) is given by \(\binom{|\bar{S} \backslash \{x\}|}{b}\). Therefore
		\begin{align*}
		\chi(S;a,b) \cdot {V}^x &= \binom{|S \backslash \{x\}|}{a-|S \cap \{x\}|} 			\binom{|\bar{S} \backslash \{x\}|}{b} \\
        &= \binom{|S|-S(x)}{a - S(x)}\binom{r-|S| - 1 + S(S)}{b}.
		\end{align*}
	\end{proof}
	\begin{proof}[Proof of Proposition \ref{prop:x-y}]
	\[
	(M_r{V}^x)_S=\sum_{a, b} \left [ \binom{a}{2}+\binom{b}{2} \right ] \binom{|S|- S(x) }{a- S(x)}\binom{r-|S|-(1-S(x))}{b}
	\]
	If \(S(x)=1\), then
	\begin{align*}
		(M_r {V}^x)_S &= \sum_a \binom{a}{2}\binom{|S|-1}{a-1} \sum_b \binom{r-|S|}{b}\\ 
    	&+ \sum_a \binom{|S|-1}{a-1} \sum_b \binom{b}{2}\binom{r-|S|}{b}.
	\end{align*}
	Thus, by (\ref{eq1})
	\begin{align*}
		(M_r {V}^x)_S &= 2^{r-|S|}\sum_a \left [\binom{|S|-1}{2}\binom{|S|-3}{a-3}+(a-1)\binom{|S|-1}{a-1}\right ]\\
        &+2^{|S|-1}\sum_b \binom{r-|S|}{2}\binom{r-|S|-2}{b-2}.
	\end{align*}
	By (\ref{eq2})
	\[
	(M_r{V}^x)_S=2^{r-4}[(|S|-1)(|S|-2) + 4(|S|-1)+(r-|S|)(r-|S|-1)].
	\]
	If \(S(x)=0\), then
	\[
	(M_r{V}^x)_S = \sum_a \binom{a}{2}\binom{|S|}{a} \sum_b \binom{r-|S|-1}{2}+\sum_a \binom{|S|}{a} \sum_b \binom{b}{2}\binom{r-|S|-1}{b}.
    \]
	By (\ref{eq1})
	\begin{align*}
		(M_r{V}^x)_S&= 2^{r-|S|-1} \sum_a\binom{|S|}{2}\binom{|S|-2}{a-2}+2^{|S|} \sum_b 			\binom{r-|S|-1}{2}\binom{r-|S|-3}{b-2}\\
		&= 2^{r-4}[|S|^2-|S|+(r-|S|-1)(r-|S|-2)]
	\end{align*}
	\begin{description}
		\item[Case \(S(x) = S(y)\):] Then clearly \({V}_S=0\) and \((M_r{V})_S = 0\).
		\item[Case \(S(x)=1,S(y)=0\):] then \({V}_S=1\) and
		\[(M_r {V})_S=(M_r {V}^x)_S-(M_r {V}^y)_S= 2^{r-3}(r-2).\]
		\item[Case 3.]
		\(S(x)=0,S(y)=1\) then \({V}_S=-1\) and
		\[(M_r {V})_S=(M_r {V}^x)_S-(M_r {V}^y)_S= -2^{r-3}(r-2).\]
	\end{description}
	Therefore \((M_r {V})_S = 2^{r-3}(r-2) {V}_T\) for all \(S=T\).
\end{proof}

%
%
\begin{prop} \label{prop:(w-x)(y-z)}
	Let \(\mathcal{F}_r\) be a set of vectors of the form \(V_T = (T(w)-T(x))(T(y)-T(z))\), where \(\{w,x,y,z\} \subset [r]\). Then all elements of \(\mathcal{F}_r\) are eigenvectors of \(M_r\) with eigenvalue \(2^{r-3}\).
\end{prop}

\begin{proof} 
	Define the quantity \(\mu(w,y)\) to the number of ways to choose \(T\) such that \(|T| = r - |S| + a - b\); \(w\), \(y \in T\) and \(x\), \(z \not\in T\); note that this depends on \(S\), \(a\), and \(b\), although these variables are suppressed in the notation for ease of reading. We have
	\[
		\chi(S;a,b) \cdot {V}_T=\mu (w,y) - \mu (w,z) - 		\mu(x,y) + \mu(x,z),
	\]
	where
	\[
		\mu (w,y) =\binom{|S|-f}{a-S(w)-S(y)}\binom{r-|S|+f-4}{b-f+S(w)+S(y)}
	\]
	and \(f=S(w)+S(x)+S(y)+S(z).\) Therefore 
	\[
		(M_r {V})_S = \sum_{a,b} \left ( \binom{a}{2}+\binom{b}{2} \right ) \left ( \mu (w,y) - \mu (w,z) - \mu(x,y) + \mu(x,z) \right )
	\]
	\begin{description}
		\item[Case \(f = 0\) or \(f=4\):]
			Clearly in this case \((M_r {V})_S = 0\).

		\item[Case \(f=1\):] Then
			\begin{align*}
				&(M_r{V})_S =\\
				& \sum_{a,b}\left [\binom{a}{2}+\binom{b}{2}\right ]\left [\binom{|S|-1}{a-1 }\binom{r-|S|-3}{b}+\binom{|S|-1}{a}\binom{r-|S|-3}{b-1} \right .\\
				&\left .\qquad-\binom{|S|-1}{a-1}\binom{r-|S|-3}{b}-\binom{|S|-1}{a}\binom{r-|S|-3}{b-1}\right ]=0.
			\end{align*}

		\item[Case \(f=3\):] Then
			\begin{align*}
				&(M_r{V})_T=\\
				&\sum_{a,b}\left [\binom{a}{2}+\binom{b}{2}\right ]\left [\binom{|S|-3}{a-2}\binom{r-|S|-1}{b-1}+\binom{|S|-3}{a-1}\binom{r-|S|-1}{b-2} \right .\\
				&\left . \qquad -\binom{|S|-3}{a-1}\binom{r-|S|-1}{b-2}-\binom{|S|-3}{a-2}\binom{r-|S|-1}{b-1}\right ]=0.
			\end{align*}
		\item[Case \(S(w)=S(y)=1,S(x)=S(z)=0\) or \(S(x)=S(z)=1,S(w)=S(y)=0\):] Then
			\begin{align*}
				&(M_r{V})_S=\\
				&\sum_{a,b}\left [\binom{a}{2} + \binom{b}{2} \right ] \left [ \binom{|S|-2}{a-2} \binom{r|S|-2}{b} + \binom{|S|-2}{a} \binom{r-|S|-2}{b-2} \right . \\
				&\left . \qquad -\binom{|S|-2}{a-1}\binom{r-|S|-2}{b-1}-\binom{|S|-2}{a-1}\binom{r-|S|-2}{b-1}\right ]\\
				&=P_1+P_2-2P_3
			\end{align*}
			where
			\[
            	P_1=\sum_a \binom{a}{2}\binom{|S|-2}{a-2}\sum_b\binom{r-|S|-2}{b}+\sum_a \binom{|S|-2}{a-2}\sum_b \binom{b}{2}\binom{r-|S|-2}{b}
            \]
            \[
            	P_2=\sum_a \binom{a}{2}\binom{|S|-2}{a} \sum_b \binom{r-|S|-2}{b-2} + \sum_a \binom{|S|-2}{a} \sum_b \binom{b}{2} \binom{r-|S|-2}{b-2}
            \]
			\[
            	P_3=\sum_a \binom{a}{2} \binom{|S|-2}{a-1} \sum_b \binom{r-|S|-2}{b-1} + \sum_a \binom{|S|-2}{a-1} \sum_b \binom{b}{2} \binom{r-|S|-2}{b-1}.
            \]
			By (\ref{eq1})
			\begin{align*}
				P_1 &=2^{r-|S|-2}\sum_a \left [ \binom{|S|-2}{2}\binom{|S|-4}{a-4}+ 2(a-2)\binom{|S|-2}{a-2}+\binom{|S|-2}{a-2}\right ]\\
                & + 2^{|S|-2}\sum_b\binom{r-|S|-2}{2}\binom{r-|S|-4}{b-2}\\
				&=2^{r-|S|-2}\left [\binom{s-2}{2}2^{|S|-4}+(|S|-2)2^{|S|-2}+2^{|S|-2}\right ]+\binom{r-|S|-2}{2}2^{r-4}\\
                &=2^{r-7}[(|S|-2)(|S|-3)+8|S|-8+(r-|S|-2)(r-|S|-3)].
			\end{align*}
			Similarly,
			\[
            	P_2=2^{r-|S|-4}\sum_a \binom{a}{2}\binom{|S|-2}{a}+r^{|S|-2}\sum_b \binom{b}{2}
\binom{r-|S|-2}{b-2}.
			\]
			By (\ref{eq5})
			\[
            	P_2=2^{r-7}\left [(|S|-2)(|S|-3)+(r-|S|-2)(r-|S|-3)+8(r-|S|-2)+8 \right ],
            \]
			and
			\[
            	P_3=2^{r-|S|-2}\sum_a \binom{a}{2}\binom{|S|-2}{a-1}+2^{|S|-2}\sum_b \binom{b}{2}\binom{r-|S|-2}{b-1}.
            \]
			Therefore
            \[
            	P_1+P_2-2P_3=2^{r-3}.
            \]
			\item[Case \(S(x)=S(y)=1\), \(S(w)=S(z)=0\) or \(S(w)=S(z)=0\), \(S(x)=S(y)=1\):] This case is identical to the above case except that both \((M_r{V})_S\) and \({V}_S\) are negated.
		\end{description}
		Therefore \((M_r{V})_S = 2^{r-3}{V}_T\) for all \(S=T\).
\end{proof}

%
%

Proposition \ref{prop:(r-2s)(x-y)} describes a second family of eigenvectors of \(M_r\) with eigenvalue \(2^{r-3}\). In Proposition \ref{prop:size 2^(r-3)}, we will show that this family of vectors is distinct from those described in Proposition \ref{prop:(w-x)(y-z)}.
\begin{prop} \label{prop:(r-2s)(x-y)}
	Let \(\mathcal{G}_r\) be a set of vectors indexed by \(T \subset [r]\) such that \(V=V^x-V^y\), where \(\{x,y\} \subset [r]\) and \(V^x_T=(r-2|T|)T(x)\). Then all elements of \(\mathcal{G}_r\) are eigenvectors of \(M_r\) with eigenvalue \(2^{r-3}\).
\end{prop}

	\begin{proof} Clearly,
		\[
			\chi(S;a,b) \cdot {V}^x = \sum_{\substack{T\\ S \thicksim T\\ x \in T}} (r-2|T|).
		\]
		It is straightforward to see that \(S, T \subset [r]\), \(|S \cap T| = a\) and \(|\bar{S} \cap \bar{T}| = b\) implies that \(|T| = r - |S| -a + b\). Therefore, by Lemma \ref{x-y:chi dot Vx},
		\begin{align*}
			\chi(S;a,b) \cdot {V}^x &= (r-2|T|)\sum_{\substack{T \\ S \thicksim T \\ x \in T}} 1 \\
&= \binom{|S|-S(x)}{a-S(x)}\binom{r-|S|-(1-S(x))}{b}(r-2|T|).
		\end{align*}
Then we have
	\begin{align*}
		&(M_r{V}^x)_S=\\
		&=\sum_{a,b}\left [\binom{a}{2}+\binom{b}{2}\right ]\binom{|S|-S(x)}{a-S(x)}\binom{r-|S|-(1-S(x))}{b}(r-2|T|)\\
		&=\sum_{a,b}\left [\binom{a}{2}+\binom{b}{2}\right ]\binom{|S|-S(x)}{a-S(x)}\binom{r-|S|-(1-S(x))}{b}(2|S|-r+2b-2a)\\
		&=P_1(x)+P_2(x)+P_3(x)+P_4(x)+P_5(x)+P_6(x),
	\end{align*}
	where
	\begin{align*}
		P_1(x)&=(2|S|-r)\sum_a\binom{a}{2}\binom{|S|-S(x)}{a-S(x)}\sum_b\binom{r-|S|+S(x)-1}{b}\\
		P_2(x)&=(2|S|-r)\sum_a\binom{|S|-S(x)}{a-S(x)}\sum_b\binom{b}{2}\binom{r-|S|+S(x)-1}{b}\\
		P_3(x)&=-\sum_a 2a\binom{a}{2}\binom{|S|-S(x)}{a-S(x)}\sum_b\binom{r-|S|+S(x)-1}{b}\\
		P_4(x)&=-\sum_a 2a\binom{|S|-S(x)}{a-S(x)}\sum_b\binom{b}{2}\binom{r-|S|+S(x)-1}{b}\\
		P_5(x)&=\sum_a\binom{a}{2}\binom{|S|-S(x)}{a-S(x)}\sum_b 2b\binom{r-|S|+S(x)-1}{b}\\
		P_6(x)&=\sum_a\binom{|S|-S(x)}{a-S(x)}\sum_b 2b\binom{b}{2}\binom{r-|S|+S(x)-1}{b}.
	\end{align*}
	By (\ref{eq1}),
	\begin{align*}
		P_3(x)=&2^{r-|S|+S(x)}\sum_a \left [a\binom{|S|-S(x)}{2}\binom{|S|-S(x)-2}{a-S(x)-2} \right .\\
		&\left . +aS(x)(a-S(x)) \binom{|S|-S(x)}{a-S(x)}\right ]\\
		P_3(x)=&2^{r-|S|+S(x)}\sum_a \left \{\binom{|S|-S(x)}{2}\left [ (a-S(x)-2)\binom{|S|-S(x)-2}{a-S(x)-2} \right . \right .\\
		&\left . +(S(x)+2)\binom{|S|-S(x)-2}{a-S(x)-2}\right ]+S(x)(a-S(x))^2 \binom{|S|-S(x)}{a-S(x)} \\
		& \left . + S(x)(a-S(x))\binom{|S|-S(x)}{a-S(x)} \right \}
	\end{align*}
	By (\ref{eq2}) and (\ref{eq4}),
	\begin{align*}
		P_3(x)=&-2^{r-4} \left [(|S|-S(x))(|S|-S(x)-1)(|S|+S(x)+2) \right .\\ 		& \left . +4(S(x))((|S|-S(x))+(|S|-S(x))^2)+8S(x)(|S|+S(x)) \right ].
	\end{align*}
	By similar algebra using (\ref{eq1}), (\ref{eq2}), (\ref{eq4}), and (\ref{eq5}), these terms simplify to the following.
	\begin{align*}
		P_1(x)=&2^{r-4}(2|S|-r)((|S|-S(x))(|S|-S(x)-1+4S(x)(|S|-S(x))\\
		P_2(x)=&2^{r-4}(2|S|-r)(r-|S|+S(x)-1)(r-|S|+S(x)-2)\\
		P_4(x)=&-2^{r-4}(r-|S|+S(x)-1)(r-|S|+S(x)-2)(|S|+S(x))\\
		P_5(x)=&2^{r-4}\left [ (|S|-S(x))(|S|-S(x)-1) \right .\\ 
        &\left . \qquad +4(S(x))(|S|-S(x)) \right ](r-|S|+S(x)-1)\\
		P_6(x)=&2^{r-4}(r-|S|+S(x)-1)(r-|S|+S(x)-2)(r-|S|+S(x)+1).
	\end{align*}
	Therefore
	\begin{align*}
		\frac{(M_r{V}^x)_S}{2^{r-4}} &= \left [ 6(S(x))^2+2S(x)r+r^2-4S(x)|S|-2|S|^2-6S(x)-3r+6|S|+2 \right ] \\
		&  = \left [ 2S(x)r+r^2-4S(x)|S|-2|S|^2-3r+6|S|+2 \right ],
	\end{align*}
	so
	\[
		(M_r V)_S = 2^{r-3}(S(x)-S(y))(r-2|S|).
	\]
	We may conclude that \((M_r V)_S = 2^{r-3} V_T\) for all \(S=T\).
\end{proof}


\section{Dimensions of the Eigenspaces}

\begin{prop} \label{prop:size 0}
	The dimension of the eigenspace associated with eigenvalue \(0\) is at least \(2^r- \binom{r+1}{2}\).
\end{prop}

\begin{proof}
	Recall that there are at least two families of vectors with eigenvalue 0, \(\mathcal{A}_r\) which consists solely of the vector of the form
   \[V_T = 
      \begin{cases}
         1 & |T|=r-1\\
         2-r & |T|=r\\
         0 & |T| \leq r-2.
      \end{cases}\]
	and \(\mathcal{D}_r\) which consists of vectors of the form
   \[V_T^S =
   \begin{cases}
   0 & S \not\subseteq T\\
   (-1)^{|T|} & S \subseteq T
   \end{cases}\]
	where \(S,T\subset [r]\). Let \(A\) be a square matrix indexed by subsets of \([r]\) such that \(S\) is before \(T\) in the indexing implies that \(|S|\leq |T|\) where \(S,T \subset [r]\). Let the column of \(A\) at index \(S\) be the element of \(\mathcal{D}_r\) associated with \(S\) if \(|S| < r-2\) (i.e., \(V^S\)), the single element of \(\mathcal{A}_r\) if \(S = [r-2]\) and the \(S\)-indexed column of the \(2^r \times 2^r\) identity matrix otherwise. Then, \(A\) is upper right triangular with no zero entries on the primary diagonal. This implies that \(\mathcal{A}_r \cup \mathcal{D}_r\) is linearly independent. Therefore there are at least 
	\[
    	|\mathcal{A}| + |\mathcal{D}| = 1+\sum_{i=0}^{r-3}\binom{r}{i}=2^r-\binom{r+1}{2}
    \]
	eigenvectors with eigenvalue 0.
\end{proof}


\begin{prop}
The dimension of the eigenspace of \(M_r\) associated with eigenvalue \(2^{r-3}(r-2)\) is at least \(r-1\).
\end{prop}

\begin{proof}
	Recall that there is a family of eigenvectors, \(\mathcal{E}_r\), with eigenvalue \(2^{r-3}(r-2)\). Let \(V(x,y)\) denote the element of \(\mathcal{E}_r\) whose \(S\)-coordinate is given by \(V_{S}(x,y) = S(x)-S(y)\) where \(x,y \in [r]\) and \(x \neq y\). Let \(\mathcal{E}'_r\) be the set of all vectors \(V(1,j)\) where \(j\in \{2,\ldots,r\}\). Clearly \(\mathcal{E}'_r\) is linearly independent; therefore the dimension of \(\spn(\mathcal{E}_r)\) is at least \(r-1\).
\end{proof}

\begin{remark}
	For any element \(V(x,y) \in \mathcal{E}_r\) where \(x \neq y, x \neq 1, y \neq 1\); \(V(x,y)=V(1,y)-V(1,x)\).
For any element \(V(x,1) \in \mathcal{E}_r\) where \(x \neq 1\); \(V(x,1) = -V(1,x)\). Therefore \(\mathcal{E}'_r\) spans \(\mathcal{E}\).
\end{remark}
Recall that, above, we gave two families of eigenvectors associated with eigenvalue of \(2^{r-3}\), \(\mathcal{F}_r\) and \(\mathcal{G}_r\).  Before addressing the entire eigenspace associated with \(2^{r-3}\), we address these two families separately.  Let \(G(x,y)\) denote the element of \(\mathcal{G}_r\) of the form \(G_S(x,y) = (r-2|S|)(S(x)-S(y))\) and \(F(w, x)(y, z)\) denote the element of \(\mathcal{F}_r\) of the form \(F(w,x)(y,z) = (S(w)-S(x))(S(y)-S(z))\) where \(\{w,x,y,z\} \subset [r]\).

	\begin{lemma} \label{size F'}
		There exists some \(\Gamma_r \subset \mathcal{F}_r\) such that \(\Gamma_r\) is linearly independent and \(|\Gamma_r| \geq \binom{r}{2}-r\).
	\end{lemma}

	\begin{proof}
	Let \(\gamma_S (w,x)(y,z)= (S(w)-S(x))(S(y)-S(z))\) with \(w,x,y,z \in [r]\). Let \(\Gamma_4 = \{\gamma_S(1,2)(3,4), \gamma_S(1,3)(2,4)\}\). We define \(\Gamma_r\) recursively by \(\Gamma_r = \Gamma_{r-1} \cup \Gamma'_{r}\), where
    \[
    \Gamma'_r = \left \{\gamma_S(1, j)(a, r) \, \Big | \, j \in \{2, \ldots, r-1\}, a = \min(\{2, \ldots, r-1\} \setminus \{j\}) \right \}.
    \] 
    Let \(A\) be a matrix whose first column is \(\gamma_S(1, 2)(3, r)\) and whose \(t\)-th column, for \(t \in \{2,r-2\}\), is \(\gamma_S(1, t)(2, r)\); clearly, the columns of \(A\) are \(\Gamma'_r\).  Let \(A'\) be the submatrix of \(A\) obtained by removing all except the rows \(\{1,3\}\) and \(\{2,k\}\) for \(k \in \{3,\ldots,r-1\}\).  Then \(A'\) is a diagonal matrix with all diagonal entries \(\pm 1\), so  \(\Gamma'_r\) is linearly independent. As \(\gamma_S (a,b)(c, d) \not\in \spn(\Gamma_{r-1})\) if \(r \in \{a,b,c,d\}\), \(\Gamma_r\) is, by induction, linearly independent.  Furthermore,
    \[
		|\Gamma_r| = |\Gamma_4|+\sum_{i=4}^{r-1}(i-1)=\sum_{i=0}^{r-1}(i-1)=\binom{r}{2}-r.
	\]
	\end{proof}

	\begin{lemma}\label{size G'}
		There exists some linearly independent \(\mathcal{G}'_r \subset \mathcal{G}_r\) such that \(|\mathcal{G}'| \geq r-1\).
	\end{lemma}
    
	\begin{proof}
		Let \(G(1, j) \in \mathcal{G}'_r\) be defined by \(G(1,j) = V^1 - V^j\) where \(j \in \{2, \ldots, r\}\). 
		Clearly, the elements of \(\mathcal{G}'_r\) can be arranged as columns of a matrix (with the addition of an elementary vector) such that the matrix is upper triangular and contains only non-zero entries on the principle axis.
	\end{proof}

    \begin{remark}
		It is straightforward to check that \(F(w,x)(y,z) = F(y,z)(w,x)\) and \(F(x,w)(y,z)=-F(x,w)(y,z)\); \(\mathcal{G}'_r\) is a basis of \(\spn(\mathcal{G}_r)\) via the following identities: \(G(x,y) = V_S(1,y) - V_S(1,x)\), \(G(x,1) = -G(1,x)\); and		\(\Gamma'_r\) is a spanning set for \(\Gamma_r\).
    \end{remark}
    
Now, we can combine these two results.

\begin{prop} \label{prop:size 2^(r-3)}
	The eigenspace of \(M_r\) associated with eigenvalue \(2^{r-3}\) has dimension at least \(\binom{r}{2}-1\).
\end{prop}

\begin{proof}  Let \(\mathcal{F}'_r \subset \mathcal{F}_r\) and \(\mathcal{G}'_r \subset \mathcal{G}\) where \(\mathcal{F}'_r\) and \(\mathcal{G}'_r\) be the linearly independent subsets specified in Lemma \ref{size F'} and \ref{size G'}. Let \(A\) be a matrix whose columns are the elements of \(\mathcal{G}'_r\), \(B\) a matrix whose columns are the elements of \(\mathcal{F}'_r\); and \(C = [A\;B]\) the block matrix obtained by concatenating \(A\) and \(B\).  Let \(C' = [A'\; B']\) be the submatrix of \(C\) obtained by removing all but those rows \(\mathcal{R}_1\) indexed by singletons \(\{x\}\) with \(x \in \{2,\ldots,r\}\) and a set of rows \(\mathcal{R}_2\) indexed by non-singletons so that \(B\) restricted to those rows is nonsingular.  It is possible to choose such a family of non-singletons because each element of \(\mathcal{F}'_r\) has \(T\)-coordinate \(0\) if \(|T|=1\).  Then \(C'\) is block upper-triangular and nonsingular, since the \(\mathcal{R}_1\) portion of \(A'\) is diagonal and nonsingular, the \(\mathcal{R}_1\) portion of \(B'\) is zero, and the \(\mathcal{R}_2\) portion of \(B'\) is nonsingular. Therefore \(\mathcal{F}'_r \cup \mathcal{G}'_r\) is linearly independent. By Lemmas \ref{size F'} and \ref{size G'} the dimension of the eigenspace corresponding to eigenvalue \(2^{r-3}\) has dimension at least \(\binom{r}{2}-1\).
\end{proof}

\begin{proof}[Proof of Theorem \ref{thm:main-thm}]
The description of eigenspaces as above is complete, because the dimensions add up to \(2^r\):
\begin{center}
	\begin{tabular}{|c|c|c|}
		\hline
		Eigenvalue & Families & Dimension\\
		\hline\hline
		0 &\(\mathcal{A},\mathcal{D}\) &\(2^r-\binom{r+1}{2}\)\\
		\hline
		\(2^{r-4}(r-1)(r-2)\) & \(\mathcal{B}\)&1\\
		\hline
		\(2^{r-2}(r-1)\) & \(\mathcal{C}\) &1\\
		\hline
		\(2^{r-3}(r-2)\) & \(\mathcal{E}\) &\(r-1\)\\
		\hline
		\(2^{r-3}\) & \(\mathcal{F},\mathcal{G}\) &\(\binom{r}	{2}-1\)\\
		\hline
	\end{tabular}
\end{center}
\end{proof}

\section{Appendix}

Here we collect several standard identities for use in the above calculations.  If not immediately verifiable, see \cite{S12}.

\begin{equation} \label{eq1} \binom{k}{2}\binom{n}{k-x} = \binom{n}{2}\binom{n-2}{k-x-2}+x(k-x)\binom{n}{k-x}+\binom{x}{2}\binom{n}{k-x} \end{equation}

\begin{equation} \label{eq2} \sum_k k\binom{n}{k}=n2^{n-1} \end{equation}

Chu-Vandermonde identity:

\begin{equation} \label{eq3} \sum_{j=0}^k \binom{m}{j}\binom{n-m}{k-j}=\binom{n}{k} \end{equation}

\begin{equation} \label{eq4} \sum_{k=0}^n k^2\binom{n}{k}=(n+n^2)2^{n-2} \end{equation}

\begin{equation} \label{eq5} \sum_k \binom{k}{2}\binom{n}{k-x}=\binom{n}{2}2^{n-2}+xn2^{n-1}+\binom{x}{2}2^{n} \end{equation}

\section{Acknowledgements}

Thank you to Gerhard Woeginger for suggesting the proof of Proposition \ref{prop:setsystem}.

\end{document}